\documentclass[twoside,12pt]{amsart}

\usepackage{amsmath,amsthm,amsfonts,amssymb,verbatim,enumerate}
\usepackage{array,charter}
\usepackage{longtable}
\usepackage{verbatim}
\usepackage{pdfpages}

\usepackage{latexsym,mathptm, times,mathcomp,lscape}
   \usepackage{tikz,stmaryrd}
\usepackage{calligra}
\usepackage{calrsfs}
\usepackage[mathscr]{euscript}

\input amssym.def
\input amssym 
\input xypic
\input xy
\xyoption{all}

\usepackage[right=1.5in,left=1.5in,top=1in,bottom=1in]{geometry}
\usepackage{amssymb}

\def \transpose{^{\rm T}}

\def\ev{\operatorname{ev}}

\def\F{F}

\def\mff{\mathfrak f}

\def\Q{\mathbb Q}
\def\Tot{\operatorname{Tot}}

\def\Ring{\mathfrak R}

\def\R{\goth R}

\def\ts{\textstyle}

\def\e{\epsilon}

\def\blop{\Psi}

\def\la{\langle}
\def\ra{\rangle}

\def\B{\mathbb B}

\def\Ext{\operatorname{Ext}}

\def\kk{\pmb k}
\def\im{\operatorname{im}}

\def\w{\wedge}

\def\ann{\operatorname{ann}}
\def\p{\oplus}
\def\Hom{\operatorname{Hom}}

\def\m{\mathfrak m}

\def\M{\mathbb M}
\def\t{\otimes}
\def\w{\wedge}
\def\id{\operatorname{id}}
\def\HH{\operatorname{H}}

\def\pd{\operatorname{pd}}
\def\grade{\operatorname{grade}}

\def\a{\alpha}

\def\bw{\bigwedge}
\def\g{\gamma}

\def\rank{\operatorname{rank}}

\def\Brown{\operatorname{b}}
\def\Br{\Brown}

\def\gen{\operatorname{gen}}
\def\Gen{\operatorname{Gen}}

\newtheorem{theorem}{Theorem}[section]

\newtheorem{summary-no-advance}[equation]{Summary}

\newtheorem{observation}[theorem]{Observation}
\newtheorem{observation-no-advance}[equation]{Observation}
\newtheorem{theorem-no-advance}[equation]{Theorem}
\newtheorem{proposition}[theorem]{Proposition}
\newtheorem{corollary}[theorem]{Corollary}
\newtheorem{claim-no-advance}[equation]{Claim}

\newtheorem{partial conjecture}[theorem]{Partial Conjecture}

\newtheorem{hypothesis-no-advance}[equation]{Hypothesis}

\theoremstyle{definition}
\newtheorem{definition and remark}[theorem]{Definition and Remark}

\newtheorem{data}[theorem]{Data}
\newtheorem{notation}[theorem]{Notation}

\newtheorem{definitions and conventions}[theorem]{Definition and Conventions}

\newtheorem{definition}[theorem]{Definition}
\newtheorem{definition-no-advance}[equation]{Definition}

\newtheorem{chunk}[theorem]{}
\newtheorem{remark-no-advance}[equation]{Remark}
\newtheorem{chunk-no-advance}[equation]{}
\newtheorem{example}[theorem]{Example}

\newtheorem{subchunk}[equation]{}
\newtheorem{remark}[theorem]{Remark}

\newtheorem{Earlier Notes}[theorem]{Earlier Notes}

\newtheorem{goal-no-advance}[equation]{Goal}

\newtheorem{assumption-no-advance}[equation]{Assumption}
\newtheorem{scratch work}[theorem]{Scratch Work}
\newtheorem{proposed definition}[theorem]{Proposed Definition}

\newtheorem{marching orders}[theorem]{Marching Orders}
\numberwithin{equation}{theorem}

\begin{document}

\baselineskip=16pt

\title[Perfect modules with Betti numbers $(2,6,5,1)$]
{Perfect modules with Betti numbers $(2,6,5,1)$}

\date{\today}

\author[A.~R.~Kustin]{Andrew R.~Kustin}
\address{Andrew R.~Kustin\\ Department of Mathematics\\ University of South Carolina\\\newline 
Columbia\\ SC 29208\\ U.S.A.} \email{kustin@math.sc.edu}

\subjclass[2010]{13C40, 13D02, 13H10, 13D10}

\keywords{codimension three, Cohen-Macaulay, complete intersection, cotangent cohomology, linkage class, perfect ideal, rigid algebras, strongly unobstructed algebras.}

\thanks{}

\begin{abstract}  In 2018 Celikbas, Laxmi, Kra\'skiewicz, and Weyman exhibited an interesting family of perfect ideals of codimension
three, with five generators, of Cohen-Macaulay type two with trivial multiplication on the
Tor algebra. All previously known perfect ideals of codimension
three, with five generators, of Cohen-Macaulay type two had been found by Brown in 1987. Brown's ideals all have non-trivial multiplication on the Tor algebra.
We prove that all of the ideals of Brown are 
obtained from 
the ideals of   Celikbas, Laxmi, Kra\'skiewicz, and Weyman by (non-homogeneous) specialization.
We also prove that both families of ideals, when built using power series variables over a field,  define rigid algebras in the sense of Lichtenbaum and Schlessinger.\end{abstract}

\maketitle

\section{Introduction.}
The theorems of Hilbert-Burch and Buchsbaum-Eisenbud serve as inspirations and models for structure theorems about finite free resolutions in commutative algebra. The  
Hilbert-Burch Theorem \cite{Bu68} states that if $I$ is a perfect ideal of grade two in a commutative Noetherian local ring $R$, then $I$ is generated by the maximal order minors of an $n\times (n-1)$ matrix $X$ (for some integer $n$) and $I$ is presented by $X$. The Buchsbaum-Eisenbud Theorem \cite{BE77} 
 states that if $I$ is a perfect Gorenstein ideal of grade three in a commutative Noetherian local ring $R$, then $I$ is generated by the maximal order Pfaffians of an $n\times n$ alternating matrix $X$ (for some odd integer $n$) and $I$ is presented by $X$.

Let $R$ be a commutative ring and $F: 0\to F_n\to\dots\to F_0$ be an acyclic complex of free $R$-modules of finite rank. We write that $F$ has {\em Betti number  format} $b=(b_n,\dots,b_0)$ if $\rank F_i=b_i$ for each $i$. The pair $(R,F)$ is called a {\em generic pair} for the format  $b$ if, whenever  $S$ is a commutative Noetherian ring and $G$ is an acyclic complex of free $S$-modules of format  $b$,  then there exists a ring homomorphism $R\to S$ with $G$ isomorphic to $\F\t_{R} S$.

For each plausible collection of Betti numbers $b=(b_2,b_1,b_0)$ of length two, Hochster \cite[section 7]{H75} proved that there exists a 
generic pair $(R_{\Gen},
\F_{\Gen})$ for the  format $b$.
 Hochster also proved that for complexes of length two, $R_{\Gen}$ is a Noetherian, Cohen-Macaulay, normal, graded domain, which is finitely generated as an algebra over the ring of integers. 
 In the monograph \cite{H75} Hochster encouraged the commutative algebra community to find  generic pairs for other formats and to  identify the properties of the generic rings in these formats. He also made numerous conjectures.

Bruns \cite[Thm.~1]{Br84} proved that for every plausible Betti number format $b=(b_n,\dots,b_0)$, there is a generic pair $(R_{\Gen},F_{\Gen})$. Unfortunately, however, Bruns' generic rings are in general not Noetherian. This aspect of Bruns' work is rather disappointing. It shows that some of Hochster's conjectures about generic pairs are false, without shedding any real light on the structure of finite free resolutions.

Nonetheless, 
 Weyman  
has been studying resolutions of length three since the 1980's.  In \cite{W90} he associated a Lie algebra $\mathbb L$ to each length three Betti number  format $b$.  In \cite{W90} he  also proposed 
 a candidate for the generic pair corresponding to $b$ and he proved that the candidate is the generic pair provided certain homology groups involving modules over the enveloping algebra of $\mathbb L$ are zero. The recent paper \cite{W18} is a major tour de force. The bottom line 
 is that Weyman's candidate from 25 years earlier really is the generic pair for the Betti number format $b$ for complexes of length three. The process that Weyman used to reach this conclusion is fairly amazing. He associated a graph $T$ (consisting of vertices and edges forming a capital ``T'') to each  Betti number format. There is a Kac-Moody Lie algebra $\mathfrak g$ that corresponds to $T$ and $\mathbb L$ is a subalgebra of $\mathfrak g$. The critical complexes from \cite{W90} are parts of Bernstein-Gelfand-Gelfand complexes associated to $\mathfrak g$ and therefore the  
required homology groups are automatically zero and the candidate for the generic pair for each length 3 format $b$ which is proposed in \cite{W90} is indeed the generic pair for $b$. 

Of course, extra information can be gleaned from \cite{W18}. In particular, if $(R,F)$ is a generic pair for the length three Betti number format $b$, then $R$ is Noetherian if and only if the graph associated to $b$ is a Dynkin diagram. There aren't very many Dynkin diagrams that look like a capital $T$. (There are straight lines which, according to  \cite{CVW17}, 
correspond to Gorenstein ideals and almost complete intersection ideals; and there are the three exceptional Dynkin diagrams: $E_6$, $E_7$, and $E_8$.) 
These generic resolutions that correspond to Dynkin diagrams are of particular interest!

The present paper, as well as the papers \cite{Br87} and \cite{CKLW}, are about resolutions of perfect modules of Betti number format $(2,6,5,1)$. This format corresponds to the Dynkin diagram $E_6$.

Anne Brown \cite{Br87} studied perfect ideals $I$ of grade three in the ring $R$  in a local or homogeneous setting. If the Betti numbers in the minimal resolution of $R/I$ by free $R$-modules are
\begin{equation}\label{0}0\to R^2\to R^6\to R^5\to R\end{equation}
and at least one of the Koszul relations on the minimal generators of $I$ is a minimal relation on the generators of $I$, then Brown described all of the differentials in (\ref{0}).

The resolution in \cite{CKLW} is obtained using the generic ring technique of Weyman from \cite{W18}. In the resolution of \cite{CKLW} none of  the Koszul relations on the minimal generators of $I$ are minimal relations on the generators of $I$. A homogeneous (or local) specialization will not produce units where units did not previously appear. At the ``Structure of length 3 resolutions workshop'' at the University of California, San Diego in August 2019, Weyman asked if some non-homogeneous specialization of the most general  resolution from \cite{CKLW} is equal to the most general  resolution from \cite{Br87}.
We answer Weyman's question in the affirmative in Theorem~\ref{yes}.  

The resolution  of \cite{CKLW} only makes sense when two is a unit. The resolution $\Q$ that we construct in Definition~\ref{K} makes sense over any base ring; in particular, $\Q$ may be constructed over the ring of integers. 
Furthermore, $\Q$ is isomorphic to the resolution of \cite{CKLW} when two is a unit; see Theorem~\ref{K-to-CKLW}.

The recipe for creating the  resolution $\Q$ of Definition~\ref{K} is given (but was not used) in \cite{CKLW}. The resolution in \cite{CKLW} was obtained using methods from Geometry and  
Representation Theory,  
followed by a deformation.  Once the resolution was obtained, then the linkage history of the resulting ideal was found. Linkage is a symmetric operation. The authors  of \cite{CKLW} linked from the ideal with Betti number format $(2,6,5,1)$ to a complete intersection. We used their instructions to link from a complete intersection to obtain the resolution of the most general ideal that is studied in \cite{CKLW}. 
 We do not need to deform; indeed, the expression in our approach that corresponds to the deformation variable in \cite{CKLW} is fairly complicated  and this explains why our ultimate formulas are 
 less complicated than the formulas of \cite{CKLW}.

We used Brown's Theorem to produce our specialization from the resolution of \cite{CKLW} to the resolution of \cite{Br87}. As a first approximation, we set ``$y_{12}$'' equal to $1$ and made no other change; this step was proposed by Weyman. It is clear from the multiplicative structure on the resolution of \cite{CKLW} 
 that the zeroth homology of the resulting complex satisfies the hypothesis of Brown's Theorem and therefore can be resolved by Brown's resolution. 
We used invertible row and column operations to carefully maneuver the  resolution of \cite{CKLW}, with $y_{12}$ set equal to one, into Brown's form. Most of the variables involved in Brown's ideal
appear as linear forms in the middle differential. We were able to read the rest of  an appropriate specialization from this information. 

When the resolutions of \cite{Br87} and \cite{CKLW} are built as generally as possible, they are  candidates for generic resolutions. Thus, the zeroth homologies of these resolutions should be rigid in the sense that they can not be deformed in a non-trivial manner. We use linkage theory to verify this fact in Section~\ref{Rigidity}.

\section*{Acknowledgment}
This paper answers a question posed by Jerzy Weyman at 
 the ``Structure of length 3 resolutions workshop'' at the University of California, San Diego in August 2019. Weyman also proposed the first step toward an answer.
Preliminary versions of this calculation were made at the workshop in San Diego with Federico Galetto, Taylor Ball, and Thomas Grubb.

\tableofcontents

\section{Notation, conventions, and elementary results.}\label{Prelims}

\begin{chunk}
Unless otherwise noted, $R$ is a commutative Noetherian ring and all functors are functors of $R$-modules; that is, $\otimes$, $\operatorname{Hom}$, $(\underline{\phantom{X}})^*$, 
 and 
$\bigwedge^i$ mean
$\otimes_R$, $\operatorname{Hom}_R$, $\operatorname{Hom}_R(\underline{\phantom{X}},R)$,  
 and $\bigwedge^i_R$,  respectively.\end{chunk} 

\begin{subchunk}
A complex $\mathcal C: \cdots\to C_2\to C_1\to C_0\to 0$ of $R$-modules is called {\it acyclic} if $\HH_j(\mathcal C)=0$ for $1\le j$.\end{subchunk}

\begin{subchunk}We use 
``$\im$'' 
 as an abbreviation for ``image''. 
\end{subchunk}

\begin{subchunk}
\label{2.1.3}
 If $M$ is a matrix (or a homomorphism of free $R$-modules), then $I_r(M)$ is the ideal generated by the
$r\times r$ minors of $M$ (or any matrix representation of $M$). 
 If $M$ is a matrix, then $M\transpose$ is the transpose of $M$. If $M$ is a square matrix, then $\det M$ and $|M|$ both represent the determinant of $M$.   The square matrix $M=(m_{ij})$ is called an {\it alternating matrix} if $M+M\transpose=0$ and $m_{ii}=0$ for all $i$.\end{subchunk}

\begin{subchunk} 
If $I$ and $K$ are ideals in a ring $R$, then $K:I=\{r\in R|rI\subseteq K\}$.
\end{subchunk}

\subsection{Grade and perfection}
\begin{chunk}
Let $I$ be a proper ideal in a commutative Noetherian ring $R$.
The {\it grade} of   
$I$   
is the length of a maximal  $R$-regular sequence in $I$. (The unit ideal $R$  of $R$ is regarded as an ideal  of infinite grade.)
\end{chunk}

\begin{chunk}Let $M$ be a non-zero finitely generated module over a Noetherian ring $R$, 
$\ann(M)$ be the annihilator of $M$, and $\pd_RM$ be the projective dimension of $M$. 
If 
 $$\operatorname{grade}\ann (M)= \operatorname{pd}_R M,$$
then $M$ is called a {\it perfect $R$-module}.
\end{chunk}

\begin{chunk} Let $I$ be a proper ideal in a commutative Noetherian ring $R$.
The ideal $I$ is called a {\it perfect ideal} if $R/I$ is a perfect $R$-module.
The ideal $I$ is a {\it complete intersection} if $I$ is generated by an $R$-regular sequence.
   A perfect ideal $I$ of grade $g$ is a {\it Gorenstein ideal}
if
$\operatorname{Ext}^g_R(R/I,R)$ is a cyclic R-module. A perfect ideal $I$ of grade $g$ is an {\it almost complete intersection} if $I$ is not a complete intersection but $I$ can be generated by $g+1$ elements. 
The 
ideal $I$  is said to satisfy a property {\it generically} if $I_Q$
satisfies the property for every associated prime $Q$ of $R/I$.
\end{chunk}

\subsection{Multilinear algebra}
\begin{chunk}
\label{2.3} Let $V$ be a free module of finite rank $d$ over the commutative Noetherian ring 
$R$.  
Consider the evaluation map $\operatorname{ev}: V\otimes V^*\to R$ and let 
$$\operatorname{ev}^*:R\to V^*\otimes V$$ be the dual of $\operatorname{ev}$. Both of these $R$-module homomorphisms are completely independent of coordinates. We refer to $\ev^*(1)$ as the {\it canonical element} of $V^*\t V$. Notice that if $x_1,\dots,x_d$ and $x_1^*,\dots,x_d^*$ are dual bases for $V$ and $V^*$, respectively, then 
$$\sum_{i=1}^{d} x_i^*\t x_i$$
is the canonical element of $V^*\t V$.
In particular, 
 if
  $\chi_{V^*}$ is a basis  for $\bigwedge^{d}V^*$ and $\chi_V$ is the corresponding dual basis for $\bigwedge^{d}V$, then 
$\chi_{V^*}\otimes\chi_V$ is the canonical element of
$\bigwedge^{d}V^*\otimes\bigwedge^{d}V$. 
\end{chunk}

\begin{chunk} \label{87Not1} 
Our complexes are described in a coordinate-free manner. Let $V$ be a free module of finite rank $d$ over the commutative Noetherian ring 
$R$.  
We make much use of 
the fact that the exterior algebras
$\bigwedge^{\bullet} V$ and $\bigwedge^{\bullet} V^*$ are modules over one another. In particular, if $a_i\in \bw^iV$ and $\a_i\in \bw^iV^*$, then 
\begin{equation}\label{2.4.1} a_i(\a_i)=\a_i(a_i)\in R.\end{equation}
An $R$-module homomorphism $f:V\to V^*$ is an {\it alternating map} if $(f(a_1))(a_1)=0$ for all $a_1\in V$; furthermore, $f$ is an alternating map if and only if there is an element $\a_2\in \bw^2V^*$ such that $f(a_1)=a_1(\a_2)$ for all $a_1$ in $V$. \end{chunk}

The following   facts about the interaction of the module structures of 
$\bigwedge^\bullet V $ on  $\bigwedge^\bullet V^* $ and $\bigwedge^\bullet  V^* $ on  $\bigwedge^\bullet V $   
 are well known; see \cite[section 1]{BE75}, \cite[Appendix]{BE77},  and \cite[section 1]{Ku93}.

\begin{proposition} 
\label{A3} Let $V$ be a free module of rank $d$ over a commutative
Noetherian ring $R$ and let $b_r\in \textstyle \bigwedge^{r}V$, $c_p\in \textstyle \bigwedge^{p}V$, and $\alpha_q\in\textstyle \bigwedge^{q}(V^{*})$. 
\begin{enumerate}[\rm(a)]

\item\label{A3.a} If $r =1$, then 
$ (b_r(\alpha_q))(c_p)=b_r\wedge (\alpha_q(c_p))+(-1)^{1+q}\alpha_q(b_r\wedge c_p)$. 
\item\label{A3.b} If $q=d$, then
$(b_r(\alpha_q))(c_p)=(-1)^{(d -r)(d-p)}(c_p(\alpha_q))(b_r)$. 
\item\label{A3.c} If $p=d$, then $[b_r(\alpha_q)](c_p)= b_r\wedge \alpha_q(c_p)$.
 \item\label{A3.d} If $\Psi:V\to V'$ is a homomorphism of free $R$-modules and $\delta_{s+r}\in \textstyle \bigwedge^{s+r}({V'}^*)$, then 
$$(\textstyle \bigwedge^{s}\Psi^{*})[((\textstyle \bigwedge^{r}\Psi)(b_{r}))(\delta_{s+r})]=
b_{r}[(\textstyle \bigwedge^{s+r}\Psi^{*})(\delta_{s+r})]. \qed$$  \end{enumerate}\end{proposition}

\begin{example}Let $V$ be a free module of rank $d$ over a commutative
Noetherian ring $R$, $\sum x_i^*\t x_i$ be the canonical element of $V^*\t V$ in the sense of \ref{2.3}, $y_1,\dots,y_d$ be elements of $V$, and  $\Theta$ be an element of $\bw^dV^*$. Then
\begin{equation}(y_1\w\cdots\w y_d)(\Theta)\cdot \sum_i x_i^*\t x_i=\sum_i (-1)^{i+1} (y_1\w\dots\w\widehat{y_i}\w\dots\w y_d)(\Theta)\t y_i.\label{2.6.1}\end{equation}
\begin{proof}
Let $\theta$ be an arbitrary element of $V^*$. Observe that when $1\t \theta$ is applied to the left side of (\ref{2.6.1}) one obtains
 $$(y_1\w\cdots\w y_d)(\Theta)\cdot \sum_i \theta(x_i)\cdot x_i^* =(y_1\w\cdots\w y_d)(\Theta)\cdot  \theta$$
and when $1\t \theta$ is applied to the right side of (\ref{2.6.1}) one obtains
$$\big(\theta(y_1\w\dots\w y_d)\big)(\Theta)=(y_1\w\cdots\w y_d)(\Theta)\cdot  \theta.$$
The final equality is due to either (\ref{A3.a}) or 
(\ref{A3.c}) of Proposition~\ref{A3}.
\end{proof}
\end{example}   

\subsection{Linkage}
\begin{chunk}
\label{PLEASE WRITE ME}
If $I$ is a perfect ideal of grade $g$ in a commutative Noetherian ring $R$ and $K$ is a complete intersection ideal which is contained in $I$, then $K:I$ is a perfect ideal of grade $g$. Furthermore, if $\mathbb K$ is the Koszul complex which resolves $R/K$, $\mathbb F$ is a resolution of $R/I$ of length $g$, and $\a:\mathbb K\to \mathbb F$ is a map of complexes which extends the identity map in degree zero, then the dual of the mapping cone of $\a$ is a resolution of $R/(K:I)$. These results are due to Peskine and Szpiro (see \cite[Props.~1.3 and 2.6]{PS} or \cite
[Props.~5.1 and 5.1a]{BE77}). The results were originally stated under the hypothesis that $R$ is Gorenstein and local. Golod \cite{G80} proved that the results continue hold when  $R$ is an arbitrary commutative Noetherian ring.

The concept of perfection is particularly useful because of the   ``Persistence of Perfection Principle'', which is also known as  the ``transfer of perfection''; see \cite[Prop. 6.14]{H75} or \cite[Thm. 3.5]{BV}.

\begin{theorem-no-advance}
\label{PoPP}
Let $R\to S$ be a homomorphism of Noetherian rings, $M$ be a perfect   $R$-module,  and   $\mathbb P$ be a resolution of $M$ by projective $R$-modules. If $S\otimes_RM\neq 0$ and $$\grade (\ann M)\le \grade (\ann(S\otimes_RM)),$$ 
then $S\otimes_RM$ is a perfect $S$-module with $\pd_S(S\otimes_RM)=\pd_RM$ and  $S\otimes_R\mathbb P$ is a  resolution of $S\otimes_RM$ by projective $S$-modules.
\end{theorem-no-advance}
\end{chunk}

\begin{definition}
\label{Feb11}  Let $I$  be a perfect ideal in a commutative Noetherian ring $R$.   The {\it linkage class} 
 of $I$ is the set of all  ideals $J$ in $R$ which can be obtained from $I$ by a finite number of links. The ideal $I$ is  
 {\it licci} if $I$ is in the linkage class of a complete intersection.
\end{definition}

\subsection{Regular sequences}

\begin{proposition}
\label{7.3} 
Let $R$ be a domain, $(a_1,\dots,a_n)$ be a proper  ideal of $R$ of grade $g$, $c_1$ and $c_2$ be an integers with $c_1+c_2\le g$,  $X=(x_{ij})$ and $Y=(y_{ij})$ be  
$(n-c_1)\times c_1$ and $(n-c_1)\times c_2$, respectively,  
matrices of indeterminates, $S$ be the polynomial ring $S=R[\{x_{ij},y_{ij}\}]$, and $b$ be the product matrix
$$b=\bmatrix a_1&\cdots &a_n\endbmatrix \left[\begin{array}{c|c} X&Y\\\hline I_{c_1}&0_{c_1\times c_2}\end{array}\right],$$ where $I_{c_1}$ is the $c_1\times c_1$ identity matrix. Then the following statements hold. \begin{enumerate}[\rm(a)]
\item\label{7.3.a} If $c_1+c_2\le g-1$, then the entries of $b$ form a regular sequence on $S$ and generate a prime ideal of $S$.
\item\label{7.3.b} If $c_1+c_2\le g$,  then the entries of $b$ form a regular sequence on $S$.
\end{enumerate}
\end{proposition}

\begin{proof} Let $I$ be the ideal $(a_1,\dots,a_n)R$. 
If $c_1=0$, then the assertion is \cite[Props.~21 and 22]{HE}.
Furthermore, once one proves the result for an arbitrary $c_1$, with $0\le c_1\le g$, then one may apply \cite[Props.~21 and 22]{HE} to the ideal 
$(a_1,\dots,a_n)$ in the ring $$\frac R{I_1\left(\bmatrix a_1&\cdots &a_n\endbmatrix \bmatrix  X\\\hline I_{c_1}\endbmatrix\right)}$$in order to draw the stated conclusion because 
this ideal is generated by $$(a_1,\dots, a_{n-c_1})$$ and has grade $g-c_1$. 
Consequently, it suffices to prove the result for $c_2=0$. 

Assume, henceforth, that $c_2=0$.
The proof is by induction on $c_1$.
We first treat the case $c_1=1$. 
Let $y_1,\dots,y_n$ be indeterminates and $S_1$ be the polynomial ring $S_1=R[y_1,\dots,y_n]$.
Apply
 \cite[Props.~21 and 22]{HE} to see that 
$$\begin{cases}
\text{
$\sum\limits_i a_i y_i$ is a regular element of $S_1$,}&\text{if $1\le g$, and}\\[10pt] \text{$\sum\limits_i a_i y_i$ generates a prime ideal of $S_1$,}&\text{if $1\le g-1$}.\end{cases}$$ The element $y_n$ of $S_1$ is not an element of the ideal of $S_1$ generated by $\sum_i a_i y_i$, because $I$ is a proper ideal of $I$. It follows that 
$$\begin{cases}
\text{$\sum\limits_i a_i y_i$ is a regular element of $S_1[\frac 1{y_n}]$,}&\text{if $1\le g$, and}\\[10pt]
\text{$\sum\limits_i a_i y_i$ generates a prime ideal of $S_1[\frac 1{y_n}]$},&\text{if $1\le g-1$.}\end{cases}$$ On the other hand,
$$\frac{S_1[\frac 1{y_n}]}{(\sum_i a_i y_i)}=
\frac{R[\frac{y_1}{y_n},\dots,\frac{y_{n-1}}{y_n}]}
{\left(\sum\limits_{i=1}^{n-1} a_i \frac {y_i}{y_{n}}+ a_n\right)}\left[y_n
,\frac 1{y_n}\right].$$
We conclude that 
$$\begin{cases}
\text{$\sum\limits_{i=1}^{n-1} a_i \frac {y_i}{y_{n}}+ a_n$ 
is a regular element of $R[\frac{y_1}{y_n},\dots,\frac{y_{n-1}}{y_n}]$,
}
&\text{if $1\le g$, and }\\[10pt]
\text{$\sum\limits _{i=1}^{n-1} a_i \frac {y_i}{y_{n}}+ a_n$ generates a prime ideal in  $R[\frac{y_1}{y_n},\dots,\frac{y_{n-1}}{y_n}]$,}&\text{if $1\le g-1$.}\\
\end{cases}$$
Apply a change of variables to obtain the assertion when $c_1=1$.

Notice that the ideal $$\ts I\frac{R[y_{1c_1},\dots y_{n-1,c_1}]}{\left(\sum\limits_{i=1}^{n-1} a_i y_{ic_1}+ a_n\right)}$$ is generated by $(a_1,\dots,a_{n-1})$; consequently, when one 
considers this ideal, the parameters $c_1$, $g$, and $n$ all have decreased by one. Induction yields that  
\begin{equation}\sum\limits_{i=1}^{n-c_1}a_ix_{i1}+a_{n-c_1+1},\quad \dots,\quad  \sum\limits_{i=1}^{n-c_1}a_ix_{i,c_1-1}+a_{n-1},\quad \sum\limits_{i=1}^{n-1} a_i y_{ic_1}+ a_n\label{7.3.1}\end{equation} 
$$\begin{cases}
\text{is a regular sequence of $R[\{x_{ij}\},\{y_{ic_1}\}]$,}&\text{if $c_1\le g$, and}\\ 
\text{generates a prime ideal of $R[\{x_{ij}\},\{y_{ic_1}\}]$},&\text{if $c_1\le g-1$.}
\end{cases}$$
One can add the appropriate linear combination of the first $c_1-1$ entries of (\ref{7.3.1}) to the last entry of (\ref{7.3.1}) to remove $y_{n-c_1+1,c_1},\dots,y_{n-1,c_1}$ from the last entry of (\ref{7.3.1}), without changing the ideal generated by the entries of (\ref{7.3.1}), at the expense of complicating the multipliers of $a_1,\dots,a_{n-c_1}$ in the last entry of (\ref{7.3.1}). One further change of variables produces the assertion of the proposition on the nose.
\end{proof}

\begin{observation}
\label{7.4}
 Let $R$ be a ring, $f$ be a polynomial in $R[x]$, and $r$ be an element of $R$. If $f(r)$ is a regular element of $R$, then $f$ is a regular element of $R[x]$. \end{observation}
\begin{proof}
 Write $f$ as a polynomial in the symbol $x-r$; in particular, the constant term of $f$ in this expression is $f(r)$. Recall that if a polynomial is a zero-divisor, then its coefficients generate an ideal of grade zero; see, for example, \cite[6.13 on page 17]{Na}. 
\end{proof}

\begin{example}
\label{2.12}
If $(a_1,a_2,a_3,a_4)$ is a proper ideal of grade at least three in a domain $R$ and $x_1,x_2,x_3,x_4$ are indeterminates, then the entries of the product matrix 
$$\bmatrix a_1&a_2&a_3&a_4\endbmatrix \bmatrix x_1&x_2&x_3\\1&0&0\\0&1&0\\0&0&x_4\endbmatrix$$ form a regular sequence in the polynomial ring $R[x_1,x_2,x_3,x_4]$. \end{example}
\begin{proof} Apply Proposition~\ref{7.3} to see that the entries of
$$\bmatrix a_1&a_2&a_3&a_4\endbmatrix \bmatrix x_1&x_2&x_3\\1&0&0\\0&1&0\\0&0&1\endbmatrix$$ form a regular sequence in $R[x_1x_2,x_3]$. The assertion now follows from Observation~\ref{7.4}.\end{proof}

\section{The complex $\Q$.}
A coordinate-free description of $\Q$ is given in Definition~\ref{K}. A matrix version of $\Q$ may be found in \ref{Q-matrix}.
\begin{data}
\label{3.1} Let $R$ be a commutative Noetherian ring, $F$ and $G$ be free $R$-modules with $\rank F=3$ and $\rank G=2$, $\phi:F\to G$ and $\ell:F^*\to G$ be  $R$-module homomorphisms, $\psi\in\bw^2F^*$, $\zeta\in \bw^2G$, and $z_2\in R$. Let $p:\bw^2G\t \bw^3F^*\to R$ be a fixed isomorphism. 
\end{data}

\begin{notation}
\label{3.2} Adopt the data of \ref{3.1}.
We use $g_i$, $\g_i$, $f_i$, and $\mff_i$ for arbitrary elements of $\bw^iG$, $\bw^iG^*$, $\bw^i F$, and $\bw^iF^*$, respectively. We use $\chi_F\t \chi_{F^*}$ and $\chi_G\t \chi_{G^*}$ for the canonical elements of $\bw^3F\t \bw^3F^*$ and $\bw^2G\t \bw^2G^*$, respectively; see \ref{2.3}. Let $\sum_{i=1}^3 E_i\t E_i^*$ be the canonical element of $F\t F^*$. Notice that every expression which involves $\chi_F$ must also involve  $\chi_{F^*}$. The analogous statement is also in effect for each of the objects $\chi_{F^*}$, $\chi_G$, and  $\chi_{G^*}$.  
\end{notation}

\begin{definition}
\label{K} 
 Adopt Data {\rm\ref{3.1}} and Notation~{\rm\ref{3.2}}.
Let $\Q$ be the following collection of $R$-module homomorphisms: 
$$0\to Q_3\xrightarrow{q_3}Q_2 \xrightarrow{q_2}Q_1\xrightarrow{q_1}Q_0,$$
where
\begingroup\allowdisplaybreaks
\begin{align*}\ts &\ts Q_3=G^*,\quad  Q_2=\bw^3F^*\p G\p F^*, \quad Q_1=(\bw^2G\t \bw^3F^*)\p\bw^3F^*\p\bw^2F,\\
 &\ts Q_0=\bw^3F,\\
&q_3(\g_1)=\bmatrix \phi^*(\g_1)\w \psi\\
(\phi\circ \ell^*)(\g_1)-\g_1(\zeta)\\
[\ell^*(\g_1)](\psi)+z_2\cdot \phi^*(\g_1)\endbmatrix,\\
&q_2\left(\bmatrix \mff_3\\0\\0\endbmatrix\right)=
\bmatrix
\zeta\t \mff_3
\\\hline
-z_2\cdot \mff_3  \\\hline
-p(\chi_G\t \mff_3)\cdot (\bw^2 \ell^*)(\chi_{G^*}) 
\\
\endbmatrix,
\\&q_2\left(\bmatrix 0\\g_1\\0\endbmatrix\right)=
\bmatrix
-[g_1\w  \phi(\psi(\chi_F))]\t \chi_{F^*}
\\\hline
0 \\\hline
-p(\chi_G\t \chi_{F^*})\cdot
[\ell^*(g_1(\chi_{G^*}))]\w\psi(\chi_F)
\\
-z_2\cdot p(\chi_G\t \chi_{F^*})\cdot [\phi^*(g_1(\chi_{G^*}))](\chi_F)
\endbmatrix, \\
&q_2\left(\bmatrix 0\\0\\\mff_1\endbmatrix\right)=
\bmatrix
+\chi_G\t \big(\mff_1\w [(\bw^2\phi^*)(\chi_{G^*})]\big)
\\\hline
+(\mff_1\w\psi)  \\\hline
-p\big (\sum_i[\phi(E_i)\w \ell (E_i^*)]
\t \chi_{F^*}\big)\cdot \mff_1(\chi_F)
\\
+ p(\chi_G\t \chi_{F^*})\cdot 
\Big[\phi^*\Big([\ell(\mff_1)](\chi_{G^*})\Big)\Big](\chi_F)
\\
-  p(\zeta\t \chi_{F^*})\cdot  \mff_1(\chi_F) 
\endbmatrix,\\
&q_1\left(\bmatrix
g_2\t \mff_3\\ 0\\0
 \endbmatrix\right)
=\begin{cases}
-p(g_2\t \mff_3)\cdot p((\bw^2\ell)(\psi)\t \chi_{F^*}) \cdot \chi_F\\
-z_2\cdot p(g_2\t \mff_3)\cdot p\left(
\sum_i[\phi(E_i)\w \ell (E_i^*)]\t \chi_{F^*})\right)\cdot \chi_F \\
-z_2\cdot p(g_2\t \mff_3)\cdot p(\zeta\t \chi_{F^*})\cdot \chi_F,
\end{cases}\\
&q_1\left(\bmatrix
0\\ \mff_3\\0
 \endbmatrix\right)
=\begin{cases}
+p(\chi_G\t \mff_3)\cdot p\Big([\ts\bw^2(\phi\circ \ell^*)]
(\chi_{G^*})\t \chi_{F^*}\Big)
\cdot \chi_F\\
-p(\zeta\t \chi_{F^*})\cdot p\left(\sum_i[\phi(E_i)\w \ell(E_i^*)]\t \mff_3\right) \cdot \chi_{F}\\
- p(\zeta\t \chi_{F^*})
\cdot p(\zeta\t \mff_3)\cdot \chi_F,
\end{cases}\\\intertext{and}
&q_1\left(\bmatrix
0\\ 0\\f_2
 \endbmatrix\right)
=
\begin{cases}
+p(\chi_G\t \chi_{F^*})\cdot f_2\w \ell^*\Big([\phi(\psi(\chi_F))](\chi_{G^*})\Big)\\
-p(\chi_G\t \chi_{F^*})\cdot 
z_2\cdot f_2\w [(\ts\bw^2\phi^*)(\chi_{G^*})](\chi_F)\\
-p(\chi_G\t \chi_{F^*})\cdot \zeta(\chi_{G^*})\cdot f_2\w \psi(\chi_F).\end{cases}
\end{align*}\endgroup
\end{definition}

\begin{remark}We have chosen to name an isomorphism $p:\bw^2G\t \bw^3F^*\to R$; do everything in a coordinate-free manner; and not name a preferred basis for any module. An alternate approach would involve naming preferred bases for $\bw^2G$ and $\bw^3F^*$ and then recording the maps of $\Q$ in terms of the preferred bases, when necessary. These two approaches are essentially equivalent; indeed, one can insist that 
$p(\chi_G\t \chi_{F^*})=1$ provided one removes the last two sentences of Notation~\ref{3.2} and one holds $\chi_G$, $\chi_{G^*}$, $\chi_{F^*}$, and $\chi_F$ fixed.  
\end{remark}

\begin{theorem}
\label{Q acyclic} The maps and modules $\Q$ 
 of Definition~{\rm\ref{K}} form a complex{\rm;} furthermore, if the ideal $\im q_1$ of $R$ has grade at least three, then $\Q$ is acyclic. \end{theorem}

\begin{proof} 
A routine calculation, using (\ref{2.4.1}), Proposition~\ref{A3}, and (\ref{2.6.1}), shows that $\Q$ is a complex. If the data of \ref{3.1}
is sufficiently general (for example, if $\phi$, $\ell$, $\psi$, $\zeta$, and $z_2$ are all represented by matrices of indeterminates over the ring of integers), then Theorem~\ref{6.1} shows that $\Q$ is obtained from the resolution of Brown \cite[Prop.~3.6]{Br87}
by way of linkage. Hence $\Q$ resolves a perfect module when the data of \ref{3.1} is represented by indeterminates over the ring of integers. One applies the 
``Persistence of Perfection Principle'', Theorem~\ref{PoPP}, to obtain acyclicity in the general case.
\end{proof}

\subsection{Write $\Q$ using matrices.} 
We first fix bases and write the data of \ref{3.1} in terms of matrices. \begin{notation}\label{3.5} 
Adopt the data of \ref{3.1}.
 Fix dual bases $\{e_i\}$ and $\{e_i^*\}$  for $F$ and $F^*$, respectively;
dual bases $\{\beta_i\}$ and $\{\beta_i^*\}$   for $G$ and $G^*$, respectively; dual bases $$\Omega_F=e_1\w e_2\w e_3\quad  \text{and}\quad \Omega_{F^*}=e_3^*\w e_2^*\w e_1^*$$  for $\bw^3F$ and $\bw^3F^*$, respectively; and dual bases
$$\Omega_G=\beta_1\w \beta_2\quad\text{and}\quad \Omega_{G^*}=\beta_2^*\w \beta_1^*$$ for $\bw^2G$ and $\bw^2G^*$, respectively. 
Let $\phi$ and $\ell$ have matrices 
$$A=\bmatrix a_{11}&a_{12}&a_{13}\\ a_{21}&a_{22}&a_{23}\endbmatrix\quad\text{and}\quad L=\bmatrix \ell_{11}&\ell_{12}&\ell_{13}\\ \ell_{21}&\ell_{22}&\ell_{23}\endbmatrix$$ with respect to the above bases; in particular,
$$\phi(e_j)=a_{1j}\beta_1+a_{2j}\beta_2\quad\text{and}\quad \ell(e_j^*)=\ell_{1j}\beta_1+\ell_{2j}\beta_2,$$for $1\le j\le 3$. Let
 $$\psi=\psi_{12}e_1^*\w e_2^*+\psi_{13}e_1^*\w e_3^* +\psi_{23}e_2^*\w e_3^*$$ and $\zeta=z_1 \beta_1\w \beta_2$. Let $P$ be the vector 
$$P=\bmatrix \psi_{23}\\-\psi_{13}\\ \psi_{12}\endbmatrix.$$
Define $p(\Omega_G\t \Omega_{F^*})=1$.
\end{notation}
\begin{definition}
\label{3.5.1} Retain the notation of {\rm\ref{3.5}}. If
$$S=\{a_{ij}|1\le i\le 2, 1\le j\le 3\}\cup\{\ell_{ij}|1\le i\le 2, 1\le j\le 3\}\cup \{\psi_{12},\psi_{13},\psi_{23},z_1,z_2\}$$are indeterminates 
over the commutative Noetherian  ring $R_0$ and $R$ is the polynomial ring $R_0[S]$, then one says that $\Q$ is {\it built using variables over} $R_0$ and one refers to $R$ as $R_0[\Q]$.
\end{definition}

\begin{remark} 
When we give the matrices from a matrix version of $\Q$, we include the basis of the source across the top of the matrix and the basis for the target down the left hand column.\end{remark}
\begin{proposition}
\label{Q-matrix} When the notation of {\rm\ref{3.5}} is used, then the complex $\Q$ of 
Definition~{\rm\ref{K}} has the  form
$$0\to R^2\xrightarrow{\la q_3\ra} R^6\xrightarrow{\la q_2\ra}R^5\xrightarrow{\la q_1\ra} R,$$
where $$\la q_1\ra=\begin{array}{|c||c|c|c|c|c|}\hline
&\Omega_G\t\Omega_{F^*}&e_2\w e_3&-e_1\w e_3&e_1\w e_2&\Omega_{F^*}\\\hline\hline
\Omega_F&\gen_1&\gen_2&\gen_3&\gen_4&\gen_5\\\hline\end{array}\ ,$$
$\la q_2\ra=\bmatrix q_{2,\ell}&q_{2,r}\endbmatrix$, and
$$\la q_3\ra=\begin{array}{|c||c|c|}\hline
&\beta_1^*&\beta_2^*\\\hline\hline
\beta_1&a_{11}\ell_{11}+a_{12}\ell_{12}+a_{13}\ell_{13}
&a_{11}\ell_{21}+a_{12}\ell_{22}+a_{13}\ell_{23}+z_1
\\\hline
\beta_2&a_{21}\ell_{11}+a_{22}\ell_{12}+a_{23}\ell_{13}-z_1
&a_{21}\ell_{21}+a_{22}\ell_{22}+a_{23}\ell_{23}
\\\hline
\Omega_{F^*}&-a_{11}\psi_{23}+a_{12}\psi_{13}-a_{13}\psi_{12}
&-a_{21}\psi_{23}+a_{22}\psi_{13}-a_{23}\psi_{12}
\\\hline
e_1^*&-\ell_{12}\psi_{12}-\ell_{13}\psi_{13}+z_2a_{11}&-\ell_{22}\psi_{12}-\ell_{23}\psi_{13}+z_2a_{21}
\\\hline
e_2^*&\ell_{11}\psi_{12}-\ell_{13}\psi_{23}+z_2a_{12}&+\ell_{21}\psi_{12}-\ell_{23}\psi_{23}+z_2a_{22}
\\\hline
e_3^*&\ell_{11}\psi_{13}+\ell_{12}\psi_{23}+z_2a_{13}&+\ell_{21}\psi_{13}+\ell_{22}\psi_{23}+z_2a_{23}\\\hline
\end{array}\ ,$$
with
\begingroup\allowdisplaybreaks
\begin{align*}
\gen_1&{}=
-\left|\begin{matrix}
\psi_{23}&-\psi_{13}&\psi_{12}\\
\ell_{11}&\ell_{12}&\ell_{13}\\
\ell_{21}&\ell_{22}&\ell_{23}\end{matrix}\right|
-z_2\cdot \sum_{i=1}^3
\left|
\begin{matrix} a_{1i}&\ell_{1i}\\a_{2i}&\ell_{2i}\end{matrix}\right|
-z_2z_1,\\
\gen_2&{}=P\transpose A\transpose\bmatrix \ell_{21}\\-\ell_{11} \endbmatrix-z_2 \left|\begin{matrix} a_{12}&a_{13}\\a_{22}&a_{23}\end{matrix}\right|+z_1\psi_{23},\\
\gen_3&{}=P\transpose A\transpose\bmatrix \ell_{22}\\-\ell_{12}\endbmatrix+z_2 \left|\begin{matrix} a_{11}&a_{13}\\a_{21}&a_{23}\end{matrix}\right|-z_1\psi_{13},\\
\gen_4&{}=P\transpose A\transpose\bmatrix \ell_{23}\\-\ell_{13}\endbmatrix-z_2 \left|\begin{matrix} a_{11}&a_{12}\\a_{21}&a_{22}\end{matrix}\right|+z_1\psi_{12},\\
\gen_5&{}=
-\det (AL\transpose)
-z_1
\sum\limits_{i=1}^3 \left| \begin{matrix} a_{1i}&\ell_{1i}\\
a_{2i}&\ell_{2i}\end{matrix}\right| 
- z_1^2,\\
\end{align*}\endgroup
and $q_{2,\ell}$ and $q_{2,r}$ given in Tables~{\rm \ref{q_2ell}} and {\rm \ref{q_2r}}.
\end{proposition}

\begin{table}
\begin{center}
$$ 
q_{2,\ell}=\begin{array}{|c|c|c|c|}\hline
&\beta_1&\beta_2&\Omega_{F^*}
\\\hline
\Omega_G\t\Omega_{F^*}
&\begin{cases}\psi_{12}a_{23}\\-\psi_{13}a_{22}\\+\psi_{23}a_{21}\end{cases}&\begin{cases}-\psi_{12}a_{13}
\\+\psi_{13}a_{12}\\
-\psi_{23}a_{11}\end{cases}
&z_1\\\hline
e_2\w e_3&\begin{cases}-\ell_{22}\psi_{12}\\-\ell_{23}\psi_{13}\\+z_2a_{21}\end{cases}&
\begin{cases}\psi_{12}\ell_{12}\\+\psi_{13}\ell_{13}\\-z_2a_{11}\end{cases}&\left|\begin{matrix}
\ell_{12}&\ell_{13}\\
\ell_{22}&\ell_{23}
 \end{matrix}\right|
\\\hline
-e_1\w e_3&\begin{cases}\ell_{21}\psi_{12}\\-\ell_{23}\psi_{23}\\+z_2a_{22}\end{cases}&\begin{cases} -\psi_{12}\ell_{11}\\+\psi_{23}\ell_{13}\\-z_2a_{12}\end{cases}
&-\left|\begin{matrix} 
\ell_{11}&\ell_{13}\\
\ell_{21}&\ell_{23}
\end{matrix}\right|
\\\hline
e_1\w e_2&
\begin{cases}\ell_{21}\psi_{13}\\+\ell_{22}\psi_{23}\\+z_2a_{23}\end{cases}
&\begin{cases}-\psi_{13}\ell_{11}\\-\psi_{23}\ell_{12}\\-z_2a_{13}\end{cases}
&\left|\begin{matrix} 
\ell_{11}&\ell_{12}\\
\ell_{21}&\ell_{22}\\
\end{matrix}\right|
 \\\hline
\Omega_{F^*}&0&0&-z_2
\\\hline
\end{array}$$
\caption{{ The matrix $q_{2,\ell}$  
of Proposition~\ref{Q-matrix}.
\label{q_2ell}}} 
\end{center}
\end{table}

\begin{table}
\begin{center}
$$q_{2,r}=\begin{array}{|c|c|c|c|}\hline
&
e_1^*&e_2^*&e_3^*\\\hline
\Omega_G\t \Omega_{F^*}&
\left|\begin{matrix}a_{12}&a_{13}\\a_{22}&a_{23} \end{matrix}\right|&-\left|\begin{matrix}a_{11}&a_{13}\\a_{21}&a_{23} \end{matrix}\right|&+ \left|\begin{matrix} a_{11}&a_{12}\\a_{21}&a_{22} \end{matrix}\right|
\\\hline
e_2\w e_3&
\begin{cases} -\left|
\begin{matrix} a_{12}&\ell_{12}\\a_{22}&\ell_{22}\end{matrix}\right|\\[15pt]
-\left|
\begin{matrix} a_{13}&\ell_{13}\\a_{23}&\ell_{23}\end{matrix}\right|\\
-z_1 \end{cases}&+ \left|\begin{matrix}a_{11}&\ell_{12}\\a_{21}&\ell_{22} \end{matrix}\right|&+ \left|\begin{matrix}a_{11}&\ell_{13}\\a_{21}&\ell_{23} \end{matrix}\right| \\\hline
-e_1\w e_3&
\left|\begin{matrix}a_{12}&\ell_{11}\\a_{22}&\ell_{21} \end{matrix}\right| &\begin{cases}
- \left|
\begin{matrix} a_{11}&\ell_{11}\\a_{21}&\ell_{21}\end{matrix}\right|\\[15pt]
-\left|
\begin{matrix} a_{13}&\ell_{13}\\a_{23}&\ell_{23}\end{matrix}\right|\\
-z_1\end{cases}&\left|\begin{matrix}a_{12}&\ell_{13}\\a_{22}&\ell_{23} \end{matrix}\right|
\\\hline
e_1\w e_2&
\left|\begin{matrix} a_{13}&\ell_{11}\\a_{23}&\ell_{21} \end{matrix}\right|&+\left|\begin{matrix} a_{13}&\ell_{12}\\a_{23}&\ell_{22} \end{matrix}\right|&
\begin{cases}
- \left|
\begin{matrix} a_{11}&\ell_{11}\\a_{21}&\ell_{21}\end{matrix}\right|\\[15pt]
- \left|
\begin{matrix} a_{12}&\ell_{12}\\a_{22}&\ell_{22}\end{matrix}\right|\\
-z_1
\end{cases}
 \\\hline
\Omega_{F^*}&
-\psi_{23}&+\psi_{13}&-\psi_{12}\\\hline
\end{array}$$
\caption{{ The matrix
$q_{2,r}$ of Proposition~\ref{Q-matrix}.
\label{q_2r}}} 
\end{center}
\end{table}

\begin{remark} 
\label{3.10}
Adopt the notation of \ref{3.5}. 
If $R$ is a bi-graded ring with $$a_{ij},\ \ell_{ij}\in R_{(1,0)},\quad
z_1\in R_{(2,0)}, \quad\text{and}\quad \psi_{ij},\ z_2\in R_{(0,1)},$$ 
then  $\Q$  is the bi-homogeneous complex
$$0\to R(-5,-2)^2\xrightarrow{\ \ \la q_3\ra\ \ } \begin{matrix}R(-3,-2)^2 \\\p\\ R(-4,-1)^4\end{matrix}\xrightarrow{\ \ \la q_2\ra\ \ } \begin{matrix}R(-2,-1)^4\\\p\\ R(-4,0)\end{matrix}\xrightarrow{\ \ \la q_1\ra\ \ } R.$$
\end{remark}

\section{The complex $\M$ of Celikbas,  Laxmi, Kra\'skiewicz, and Weyman.}
We begin by recording the complex of \cite{CKLW}. We copied the second and third differentials without change and we made three necessary sign adjustments in the first differential.
\begin{definition}
\label{M} Let $R$ be a commutative Noetherian ring and let
$$\{x_{ij},y_{ij}|1\le i<j\le 4\}\cup\{z_{ijk}|1\le i<j<k\le 4\}\cup\{t\}$$ 
be elements of $R$. The complex $\M$ is 
$$0\to R^2\xrightarrow{m_3} R^6\xrightarrow{m_2} R^5\xrightarrow{m_1} R,$$
where \begingroup\allowdisplaybreaks\begin{align*}
m_1&{}=\bmatrix -u_{234}+tz_{234}&-u_{134}+tz_{134}&-u_{124}+tz_{124}&-u_{123}-tz_{123}&-u+t^2\endbmatrix,\\
m_2&{}=\bmatrix m_{2,\ell}&m_{2,r}\endbmatrix,\\
m_3&{}=\bmatrix 
\begin{cases}x_{12}y_{34}-x_{13}y_{24}+x_{14}y_{23}\\+x_{34}y_{12}-x_{24}y_{13}+x_{23}y_{14}+t\end{cases}&2\Big(x_{12}x_{34}-x_{13}x_{24}+x_{14}x_{23}\Big)\\
-2\Big(y_{12}y_{34}-y_{13}y_{24}+y_{14}y_{23}\Big)&\begin{cases}-x_{12}y_{34}+x_{13}y_{24}-x_{14}y_{23}\\-x_{34}y_{12}+x_{24}y_{13}-x_{23}y_{14}+t\end{cases}\\
-\Big(-y_{12}z_{134}+y_{13}z_{124}-y_{14}z_{123}\Big)&-\Big( -x_{12}z_{134}+x_{13}z_{124}-x_{14}z_{123}\Big)\\
-y_{12}z_{234}+y_{23}z_{124}-y_{24}z_{123}&-x_{12}z_{234}+x_{23}z_{124}-x_{24}z_{123}\\
-y_{13}z_{234}+y_{23}z_{134}-y_{34}z_{123}&-x_{13}z_{234}+x_{23}z_{134}-x_{34}z_{123}\\
-\Big(-y_{14}z_{234}+y_{24}z_{134}-y_{34}z_{124}\Big)&-\Big(-x_{14}z_{234}+x_{24}z_{134}-x_{34}z_{124}\Big)
\endbmatrix\ ,\end{align*}\endgroup
with 
\begingroup\allowdisplaybreaks\begin{align*}
u_{123}&{}=\begin{cases}-2z_{234}D(12,13)+2z_{134}D(12,23)\\-2z_{124}D(13,23)+z_{123}(D(13,24)-D(12,34)+D(14,23)),\end{cases}\\
u_{124}&{}=\begin{cases}2z_{234}D(12,14)-2z_{134}D(12,24)\\+z_{124}(D(12,34)+D(13,24)+D(14,23))-2z_{123}D(14,24),\end{cases}\\
u_{134}&{}=\begin{cases}2z_{234}D(13,14)+z_{134}(-D(12,34)-D(13,24)+D(14,23))\\+2z_{124}D(13,34)-2z_{123}D(14,34),\end{cases}\\
u_{234}&{}=\begin{cases}z_{234}(-D(12,34)+D(13,24)-D(14,23))-2z_{134}D(23,24)\\+2z_{124}D(23,34)-2z_{123}D(24,34),\end{cases}\\
u&{}=b^2-4ac,\\
\intertext{for}
D(ij,k\ell)&{}=\left|\begin{matrix} x_{ij}&x_{k\ell}\\y_{ij}&y_{k\ell}\end{matrix}\right|,\\
a&{}=x_{12}x_{34}-x_{13}x_{24}+x_{14}x_{23},\\
b&{}=x_{12}y_{34}-x_{13}y_{24}+x_{14}y_{23}+y_{12}x_{34}-y_{13}x_{24}+y_{14}x_{23},\\
c&{}=y_{12}y_{34}-y_{13}y_{24}+y_{14}y_{23},
\\
\end{align*}\endgroup
and $m_{2,\ell}$ and $m_{2,r}$ are given in Tables \ref{m_{2,ell}} and \ref{m_{2,r}}.
 \end{definition}

\begin{table}
\begin{center}
$$ 
m_{2,\ell}=\bmatrix
\begin{cases}-y_{12}z_{134}\\+y_{13}z_{124}\\-y_{14}z_{123}\end{cases}
&
\begin{cases}-x_{12}z_{134}\\+x_{13}z_{124}\\-x_{14}z_{123}\end{cases}
&\begin{cases}
-x_{12}y_{34}+x_{34}y_{12}\\
+x_{13}y_{24}-x_{24}y_{13}\\
-x_{14}y_{23}+x_{23}y_{14}\\
+t
\end{cases}
\\\hline
\begin{cases}y_{12}z_{234}\\-y_{23}z_{124}\\+y_{24}z_{123}\end{cases}
&\begin{cases}x_{12}z_{234}\\-x_{23}z_{124}\\+x_{24}z_{123}\end{cases}
&-2(x_{23}y_{24}-x_{24}y_{23})
\\\hline
\begin{cases}-y_{13}z_{234}\\+y_{23}z_{134}\\-y_{34}z_{123}\end{cases}
&\begin{cases}-x_{13}z_{234}\\+x_{23}z_{134}\\-x_{34}z_{123}\end{cases}
&2(x_{23}y_{34}-x_{34}y_{23})
\\\hline
\begin{cases}-y_{14}z_{234}\\+y_{24}z_{134}\\-y_{34}z_{124}\end{cases}&
\begin{cases}-x_{14}z_{234}\\+x_{24}z_{134}\\-x_{34}z_{124}\end{cases}
&2(x_{24}y_{34}-x_{34}y_{24})

\\\hline
0&0&-z_{234}
\endbmatrix$$
\caption{The matrix $m_{2,\ell}$ 
of Definition~\ref{M}.
\label{m_{2,ell}}} 
\end{center}
\end{table}

\begin{table}
\begin{center}
$$ 
m_{2,r}=\bmatrix
2(x_{13}y_{14}-x_{14}y_{13})
&-2(x_{12}y_{14}-x_{14}y_{12})
&-2(x_{12}y_{13}-x_{13}y_{12})
\\\hline
\begin{cases}
-x_{12}y_{34}+x_{34}y_{12}\\
-x_{13}y_{24}+x_{24}y_{13}\\
+x_{14}y_{23}-x_{23}y_{14}\\
+t
\end{cases}
&2(x_{12}y_{24}-x_{24}y_{12})
&2(x_{12}y_{23}-x_{23}y_{12})
\\\hline
2(x_{13}y_{34}-x_{34}y_{13})
&\begin{cases}
-x_{12}y_{34}+x_{34}y_{12}\\
-x_{13}y_{24}+x_{24}y_{13}\\
-x_{14}y_{23}+x_{23}y_{14}\\
-t
\end{cases}
&-2(x_{13}y_{23}-x_{23}y_{13})
\\\hline
2(x_{14}y_{34}-x_{34}y_{14})
&
-2(x_{14}y_{24}-x_{24}y_{14})
&\begin{cases}
x_{12}y_{34}-x_{34}y_{12}\\
-x_{13}y_{24}+x_{24}y_{13}\\
-x_{14}y_{23}+x_{23}y_{14}\\
+t
\end{cases}\\\hline
-z_{134}&z_{124}&z_{123}
\endbmatrix$$
\caption{ The matrix $m_{2,r}$ of Definition~\ref{M}.
\label{m_{2,r}}} 
\end{center}
\end{table}
\begin{theorem}
\label{W acyclic} The maps and modules $\M$ 
 of Definition~{\rm\ref{M}} form a complex{\rm;} furthermore, if the ideal $\im m_1$ of $R$ has grade at least three, then $\M$ is acyclic. \end{theorem}
\begin{proof}
This result is \cite[Thm.~3.1]{CKLW}. One could also use Theorems~\ref{Q acyclic},  \ref{K-to-CKLW}, and \ref{PoPP}.
\end{proof}
\begin{definition}
\label{4.2}  Retain the notation of {\rm\ref{M}}. If
$$S=\{x_{ij},y_{ij}|1\le i<j\le 4\}\cup\{z_{ijk}|1\le i<j<k\le 4\}\cup\{t\}
$$are indeterminates 
over the commutative Noetherian  ring $R_0$ and $R$ is the polynomial ring $R_0[S]$, then one says that $\M$ is {\it built using variables over} $R_0$ and one refers to $R$ as $R_0[\M]$.
\end{definition}

\begin{remark}\label{4.4} 
Adopt the language of Definition~{\rm\ref{M}}.
If $R$ is a bi-graded ring with $$x_{ij},\ y_{ij}\in R_{(1,0)},\quad
t\in R_{(2,0)}, \quad\text{and}\quad z_{ijk}\in R_{(0,1)},$$ 
then $\M$ is the bi-homogeneous complex
$$0\to R(-5,-2)^2\xrightarrow{\ \ m_3\ \ } \begin{matrix}R(-3,-2)^2 \\\p\\ R(-4,-1)^4\end{matrix}\xrightarrow{\ \ m_2\ \ } \begin{matrix}R(-2,-1)^4\\\p\\ R(-4,0)\end{matrix}\xrightarrow{\ \ m_1\ \ } R.$$
\end{remark}

\subsection{The map which carries $\Q$ onto $\M$.}
\begin{theorem}
\label{K-to-CKLW} 
 Let $R_0$ be a commutative Noetherian ring in which two is a unit.
If the complexes $\Q$ and $\M$ of Definitions {\rm\ref{K}} and {\rm\ref{M}} are both  
built using variables over 
$R_0$, in the sense of {\rm\ref{3.5.1}} and {\rm\ref{4.2}}, then there is a surjective $R_0$-algebra homomorphism $\mu:R_0[\Q]\to R_0[\M]$ such that the complexes $\Q\t_{R_0[\Q]} R_0[\M]$ and $\M$ are isomorphic. 
\end{theorem}

\begin{proof}Let $R=R_0[\M]$. Define $\mu:R_0[\Q]\to R_0[\M]$ by 
\begin{align*}
&\mu(a_{11})= -x_{12},&&
\mu(a_{12})= -x_{13},&&
\mu(a_{13})= x_{14},&&
\mu(\ell_{13})= x_{23},\\
&\mu(\ell_{12})= x_{24},&&
\mu(\ell_{11})= -x_{34},&&
\mu(a_{21})= y_{12},&&
\mu(a_{22})= y_{13},\\
&\mu(a_{23})= -y_{14},&&
\mu(\ell_{23})= -y_{23},&&
\mu(\ell_{22})= -y_{24},&&
\mu(\ell_{21})= y_{34},\\
&\mu(\psi_{12})= z_{123}/2,&&
\mu(\psi_{13})= -z_{124}/2,&&
\mu(\psi_{23})= -z_{134}/2,&&
\mu(z_2)= z_{234}/2,\\
\end{align*}and
$$\mu(z_1)=(-t+x_{12}y_{34}-x_{13}y_{24}+x_{14}y_{23}-x_{34}y_{12}+x_{24}y_{13}-x_{23}y_{14})/2.$$
Observe that
$$\xymatrix{0\ar[r]& R^2\ar[rr]^{\mu(\la q_3\ra)}\ar[d]^{\sigma_3}&&R^6\ar[rr]^{\mu(\la q_2\ra )}\ar[d]^{\sigma_2}&&R^5\ar[rr]^{\mu(\la q_1\ra)}\ar[d]^{\sigma_1}&&R\ar[d]^{\sigma_0}
\\
0\ar[r]& R^2\ar[rr]^{m_3}&&R^6\ar[rr]^{m_2}&&R^5\ar[rr]^{m_1}&&R}$$
is an isomorphism of complexes, where
$$\sigma_3=\bmatrix 0&-1\\1&0\endbmatrix,\quad 
\sigma_2=\bmatrix 
2&0&0&0&0&0\\
0&2&0&0&0&0\\
0&0&-2&0&0&0\\
0&0&0&2&0&0\\
0&0&0&0&2&0\\
0&0&0&0&0&2\endbmatrix,\quad \sigma_1=\bmatrix 4&0&0&0&0\\0&4&0&0&0\\0&0&-4&0&0\\0&0&0&4&0\\
0&0&0&0&-4\endbmatrix,$$
and 
$\sigma_0=16$.
\end{proof}

\section{The complex $\B$ of Brown.}A coordinate-free description of the  complex  $\B$ from \cite[Prop.~3.6]{Br87} is given in Definition~\ref{5.3}. A matrix version of $\B$ may be found in \ref{B-matrix}. A map which carries $\Q$ onto $\B$ is given in Theorem~\ref{yes}. 

\begin{data}
\label{5.1} Let $R$ be a commutative Noetherian ring, $F$ and $G$ be free $R$-modules with $\rank F=3$ and $\rank G=2$, $\phi:F\to G$  
be an $R$-module homomorphism, $\psi\in\bw^2F^*$, $\zeta\in \bw^2G$, and $z_2\in R$. Let $p:\bw^2G\t \bw^3F^*\to R$ be a fixed isomorphism. 
\end{data}

\begin{notation}
\label{5.2} Adopt the data of \ref{5.1}.
We use $g_i$, $\g_i$, $f_i$, and $\mff_i$ for arbitrary elements of $\bw^iG$, $\bw^iG^*$, $\bw^i F$, and $\bw^iF^*$, respectively. We use $\chi_F\t \chi_{F^*}$ and $\chi_G\t \chi_{G^*}$ for the canonical elements of $\bw^3F\t \bw^3F^*$ and $\bw^2G\t \bw^2G^*$, respectively; see \ref{2.3}. 
 Notice that every expression which involves $\chi_F$ must also involve  $\chi_{F^*}$. The analogous statement is also in effect for each of the objects $\chi_{F^*}$, $\chi_G$, and  $\chi_{G^*}$.  
\end{notation}

\begin{definition}
\label{5.3} Adopt Data {\rm\ref{5.1}} and Notation~{\rm\ref{5.2}}. 
 Let $\B$ be the following collection of $R$-module homomorphisms:
$$0\to B_3\xrightarrow{b_3}B_2\xrightarrow{b_2}B_1\xrightarrow{b_1}B_0,$$where
\begingroup\allowdisplaybreaks
\begin{align*} 
&\ts B_3=(\bw^2G^*\t \bw^3F)\p\bw^3 F,\quad
B_2= 
\bw^3 F\p G^*\p F,\quad
B_1=G\p F^*,\\
&B_0=\ts\bw^2G\t \bw^3F^*,\\
&\Br_3\left(\bmatrix \g_2\t f_3\\f_3'\endbmatrix \right)=\bmatrix \zeta(\g_2)\cdot f_3-z_2\cdot f_3'\\
-[\phi(\psi(f_3))](\g_2)\\
[(\bw^2\phi^*)(\g_2)](f_3)+\psi(f_3')\endbmatrix\ , 
\\&\Br_2\left( \bmatrix f_3\\ \g_1\\ f_1\endbmatrix\right)=\bmatrix
\phi(\psi(f_3))
+\g_1(\zeta)+z_2\cdot \phi(f_1)\\
-\phi^*(\g_1)+f_1(\psi)
 \endbmatrix\ ,\intertext{and} 
&\Br_1\left(\bmatrix g_1\\\mff_1 \endbmatrix \right)=\begin{cases}
-g_1\w \phi[\psi(\chi_F)]\t \chi_{F^*}+
z_2\cdot   \chi_G\t \mff_1\w (\ts\bw^2\phi^*)(\chi_{G^*})
\\+\zeta\t \mff_1\w \psi.\end{cases}
\end{align*}\endgroup 
\end{definition}

\begin{remark}
Our phrasing of the data involved in \ref{5.1} 
is slightly different than the phrasing in \cite[Prop.~3.6]{Br87}. Brown uses a $5\times 5$ alternating matrix $T$. Brown also points out, in the note above \cite[Prop.~3.6]{Br87}, that no generality is lost if $t_{12}$ is set equal to zero. We  view this data as a $2\times 3$ matrix $\la\phi\ra$ together with an element $\psi$ of $\bw^2F_0^*$. The element $\psi$  
corresponds to
$3\times 3$ alternating matrix $A$ as is described in \ref{87Not1} and \ref{2.1.3}. Thus, our data can be configured as
$$\left[\begin{array}{c|c}0_{2\times 2} &\la\phi\ra_{2\times 3}\\\hline\\[-10pt]
(-\la\phi\ra\transpose)_{3\times 2}& A_{3\times 3}\end{array}\right],$$which is Brown's data.\end{remark}
\begin{theorem}
\label{B is exact.} The maps and modules $\B$ 
 of Definition~{\rm\ref{5.3}} form a complex; furthermore, if the ideal $\im b_1$ of $R$ has grade at least three, then $\B$ is acyclic. \end{theorem}
\begin{proof}
This result is \cite[Prop.~3.6]{Br87}. An alternative argument which explicitly works over any commutative Noetherian base ring is sketched in Example~\ref{AltArg}. As in the proof of Theorems~\ref{Q acyclic} and \ref{W acyclic} it suffices to prove that $\B$ is acyclic when the data of \ref{5.1} is sufficiently general. 
\end{proof}

\subsection{Write $\B$ using matrices.}
We first fix bases and write the data of \ref{5.1} in terms of matrices. \begin{notation}\label{3.5FIX} 
Adopt the data of \ref{5.1}.
 Fix dual bases $\{e_i\}$ and $\{e_i^*\}$  for $F$ and $F^*$, respectively;
dual bases $\{\beta_i\}$ and $\{\beta_i^*\}$   for $G$ and $G^*$, respectively; dual bases $$\Omega_F=e_1\w e_2\w e_3\quad  \text{and}\quad \Omega_{F^*}=e_3^*\w e_2^*\w e_1^*$$  for $\bw^3F$ and $\bw^3F^*$, respectively; and dual bases
$$\Omega_G=\beta_1\w \beta_2\quad\text{and}\quad \Omega_{G^*}=\beta_2^*\w \beta_1^*$$ for $\bw^2G$ and $\bw^2G^*$, respectively. 

Let $\phi$ have matrix 
$$U=\bmatrix u_{11}&u_{12}&u_{13}\\ u_{21}&u_{22}&u_{23}\endbmatrix$$ with respect to the above bases; in particular,
$$\phi(e_j)=u_{1j}\beta_1+u_{2j}\beta_2.$$ Let 
$\zeta=w_1\beta_1\w \beta_2$, $\psi=\pi_1 e_2^*\w e_3^*-\pi_2 e_1^*\w e_3^*+\pi_3 e_1^*\w e_2^*$,
 $$\Pi=\bmatrix \pi_1&\pi_2&\pi_3\endbmatrix\transpose,$$ 
and $\Delta_i=(-1)^{i+1}$ times the determinant of $U$ with column $i$ deleted.
\end{notation}

\begin{definition}
\label{5.4} Retain the notation of {\rm\ref{3.5FIX}}. If
$$S=\{u_{ij}|1\le i\le 2, 1\le j\le 3\}\cup\{\pi_{i}|1\le i\le  3\}\cup \{w_1,z_2\}$$are indeterminates 
over the commutative Noetherian  ring $R_0$ and $R$ is the polynomial ring $R_0[S]$, then one says that $\B$ is {\it built using variables over} $R_0$ and one refers to $R$ as $R_0[\B]$.
\end{definition}

\begin{proposition}\label{B-matrix}
 When the notation of {\rm\ref{3.5FIX}} is used, then the complex $\B$ of 
Definition~{\rm\ref{5.3}} has the  form
$$0\to R^2\xrightarrow{\la b_3\ra} R^6\xrightarrow{\la b_2\ra}R^5\xrightarrow{\la b_1\ra} R,$$
where
$$\la b_3\ra=\begin{array}{|c||c|c|}\hline
&\Omega_{G^*}\t \Omega_F&-\Omega_F\\\hline\hline
\Omega_F&w_1&z_2\\\hline
-\beta_2^*&(U\Pi)_1&0\\\hline
\beta_1^*&(U\Pi)_2&0\\\hline
e_1&\Delta_1&\pi_1\\\hline
e_2&\Delta_2&\pi_2\\\hline
e_3&\Delta_3&\pi_3\\\hline\end{array}\ ,$$
$$\la b_2\ra =\begin{array}{|c|c|c|c|c|c|c|}\hline
&\Omega_F&-\beta_2^*&\beta_1^*&e_1&e_2&e_3\\\hline\hline
-\beta_2&(U\Pi)_2&0&-w_1&-z_2u_{21}&-z_2u_{22}&-z_2u_{23}\\\hline
\beta_1&-(U\Pi)_1&w_1&0&z_2u_{11}&z_2u_{12}&z_2u_{13}\\\hline
e_1^*&0&u_{21}&-u_{11}&0&-\pi_3&\pi_2\\\hline
e_2^*&0&u_{22}&-u_{12}&\pi_3&0&-\pi_1\\\hline
e_3^*&0&u_{23}&-u_{13}&-\pi_2&\pi_1&0\\\hline\end{array}\ ,$$
and $\la b_1\ra$ is equal to
$$\begin{array}{|c|c|c|c|c|c|}\hline
&-\beta_2&\beta_1&e_1^*&e_2^*&e_3^*\\\hline\hline
\Omega_G\t\Omega_{F^*}&(U\Pi)_1& (U\Pi)_2&  z_2 \Delta_1-w_1\pi_1& 
z_2 \Delta_2-w_1\pi_2& 
z_2\Delta_3-w_1\pi_3\\\hline\end{array}\ .$$
\end{proposition}

\begin{remark} \label{5.9}
Adopt the notation of \ref{3.5FIX}. 
If $R$ is a bi-graded ring with $$u_{ij} \in R_{(1,0)},\quad
w_1\in R_{(2,0)}, \quad\text{and}\quad \pi_{i},\ z_2\in R_{(0,1)},$$ 
then  $\B$  is the bi-homogeneous complex
$$0\to \begin{matrix}R(-4,-2)\\\p\\ R(-2,-3)\end{matrix}\xrightarrow{\ \ \la b_3\ra\ \ } \begin{matrix}R(-2,-2) \\\p\\ R(-3,-1)^2\\\p\\
R(-2,-2)^3
\end{matrix}\xrightarrow{\ \ \la b_2\ra\ \ } \begin{matrix}R(-1,-1)^2\\\p\\ R(-2,-1)^3\end{matrix}\xrightarrow{\ \ \la b_1\ra\ \ } R.$$
\end{remark}

\subsection{The map which carries $\Q$ onto $\B$.}
\begin{theorem}
\label{yes}
If the complexes $\Q$ and $\B$ of Definitions {\rm\ref{K}} and {\rm\ref{5.3}} are both  
built using variables over 
 the commutative Noetherian ring $R_0$, in the sense of {\rm\ref{3.5.1}} and {\rm\ref{5.4}}, then there is a surjective $R_0$-algebra homomorphism $\Phi:R_0[\Q]\to R_0[\B]$ such that the complexes $\Q\t_{R_0[\Q]} R_0[\B]$ and $\B$ are isomorphic. 
\end{theorem}
\begin{proof}
Let $R=R_0[\B]$. Define $\Phi:R_0[\Q]\to R_0[\B]$ by 
$$\Phi(a_{12})= \Phi(a_{13})=\Phi(\ell_{13})=\Phi(\ell_{12})=\Phi(\ell_{21})=0,$$
\begin{align*}
&\Phi(a_{11})=-z_2,&&
\Phi(\ell_{11})= -1,&&
\Phi(a_{21})=-w_1-\pi_1,&&
\Phi(a_{22})= u_{12},\\
&\Phi(a_{23})=  u_{13},&&
 \Phi(\ell_{23})= -u_{22},&&
\Phi(\ell_{22})= u_{23},&&
\Phi(z_1)= w_1,\\
&\Phi(\psi_{12})= \pi_{3},&&
\Phi(\psi_{13})= -\pi_{2},&&
\Phi(\psi_{23})= -u_{11},&&
\Phi(z_2)=u_{21}.\\
\end{align*}
\noindent Observe that 
\begingroup\allowdisplaybreaks
\begin{align*}
&\Phi(A)=\bmatrix -z_2&0&0\\ -w_1-\pi_1&u_{12}&u_{13}\endbmatrix,
&&\Phi(L)=\bmatrix -1&0&0\\ 0&u_{23}&-u_{22}\endbmatrix,
\\
&\Phi(P)=\bmatrix -u_{11}\\\pi_2\\ \pi_3\endbmatrix,
&&\Phi\left(\sum_{i=1}^3\left|\begin{matrix} a_{1i}&\ell_{1i}\\a_{2i}&\ell_{2i}\end{matrix}\right|\right)=-w_1-\pi_1,
\end{align*}\begin{align*}
&\Phi(AL\transpose)=\bmatrix z_2&0\\w_1+\pi_1& u_{12}u_{23}-u_{13}u_{22}\endbmatrix,\quad\text{and}\\
&\Phi(P\transpose A\transpose)=\bmatrix u_{11}z_2&u_{11}(w_1+\pi_1)+u_{12}\pi_2+u_{13}\pi_3\endbmatrix.\end{align*}\endgroup
Observe further, that 
$$\xymatrix{0\ar[r]&R^2\ar[rr]^{\la b_3\ra}\ar[d]^{J_3}&&R^6\ar[rr]^{\la b_2\ra}\ar[d]^{J_2^{-1}}&&R^5\ar[rr]^{\la b_1\ra}\ar[d]^{J_1}&&R\ar[d]^{=}\\
0\ar[r]&R^2\ar[rr]^{\Phi(\la q_3\ra)}&&R^6\ar[rr]^{\Phi(\la q_2\ra)}&&R^5\ar[rr]^{\Phi(\la q_1\ra)}&&R}$$ is an isomorphism of complexes with
$$J_1=\left[\begin{array}{cc|ccc} 0&1&0&0&0\\1&0&0&0&0\\\hline 0&0&0&-1&0\\0&0&0&0&-1\\
0&0&-1&0&0\end{array}\right], \quad 
J_2=
\bmatrix -1  &0 &0 &0  &0  &0 \\
        0   &0 &0 &-1 &0  &0 \\
        u_{11} &1 &0 &0  &0  &0 \\
        u_{21} &0 &1 &0  &0  &0 \\
        0   &0 &0 &0  &0  &1 \\
        0   &0 &0 &0  &-1 &0 \endbmatrix,\quad\text{and}$$
$$J_3=\bmatrix 0&-1\\-1&0\endbmatrix.$$

\vskip-18pt
\end{proof}

\section{The linkage history of $\Q$.}\label{6}

\begin{example}
\label{AltArg} 
Let $R_0$ be a commutative Noetherian domain and let $\B$ be the resolution
of Proposition~\ref{B-matrix} built over $R_0$ using bi-graded variables as described in Definition~\ref{5.4} and  Remark~\ref{5.9}.
We 
 give the linkage history of  
$\B$ 
 and thereby give an alternative proof of Theorem~\ref{B is exact.}. This linkage history also plays a crucial role in our discussion of rigidity in Section~\ref{Rigidity}. 
Let  \begin{equation}\label{6.1.0}R=R_0[\{\lambda_{ij}|1\le i\le 3,\ 1\le j\le 2\}\cup\{\delta_0,\a_1,\a_2,\a_3,\beta\}]\end{equation} be a 
bigraded polynomial ring.
The variables of $R$ have the following degrees:
$$\a_1,\a_2,\lambda_{ij}\in R_{10},\text{ for $1\le i,j\le 2$};\quad
\a_3\in R_{20}; \quad\text{and}\quad \beta,\delta_0,\lambda_{31},\lambda_{32}\in R_{01}.
$$ (This choice of bi-grading is compatible with the bi-grading in (\ref{3.10}), (\ref{4.4}), and (\ref{5.9}). On the other hand, it can be ignored with no damage done.)
Let $\Lambda$ be the $3\times 2$ matrix $(\lambda_{ij})$,  
$\delta=[\delta_1,\delta_2,\delta_3]$ be the row vector of signed maximal minors of $\Lambda$ with $\delta \Lambda=0$.
The ideal $(\delta_0,\delta_1,\delta_2,\delta_3)$ is a hyperplane section of a grade two almost complete intersection. This ideal is perfect of grade three and the resolution of the quotient ring it defines
\begin{equation}\label{6.1.1}0\to R(-2,-2)^2\to \begin{matrix} R(-1,-2)^2\\\p\\R(-2,-1)^3\end{matrix}
\to \begin{matrix} R(0,-1)\\\p\\ R(-1,-1)^2\\\p\\R(-2,0)\end{matrix} \to R
\end{equation}
is well known. 
Let
 $N$ be the product 
$$N=\bmatrix \delta_0&\delta_1&\delta_2&\delta_3\endbmatrix \bmatrix \a_1&\a_2&\a_3\\1&0&0\\0&1&0\\0&0&\beta\endbmatrix.$$ 
Recall, from Example~\ref{2.12}, 
that the entries of $N$ form a regular sequence, denoted $n_1,n_2,n_3$. 
The technique of linkage (see, for example, \ref{PLEASE WRITE ME}) yields that 
\begin{align*}&(n_1,n_2,n_3): (\delta_0,\delta_1,\delta_2,\delta_3)\\
{}={}&(n_1,n_2,n_3,\lambda_{11}\a_{1}\beta+\lambda_{21}\a_{2}\beta+\lambda_{31}\a_{3},
 \lambda_{12}\a_{1}\beta+\lambda_{22}\a_{2}\beta+\lambda_{32}\a_{3}).\end{align*}
Observe that  the bi-homogeneous $R_0$-algebra isomorphism $\blop:R\to R_0(\B)$ with  $\blop(\lambda_{ij})= u_{ij}$, for $1\le i,j\le 2$,
\begin{align*}
&\blop(\lambda_{31})=\pi_2,&&
 \blop(\lambda_{32})=-\pi_1,&&
\blop(\delta_0) =-\pi_3,&&
\blop(\a_1) =u_{23},&&\\
&\blop(\a_2) =-u_{13},&&
\blop(\beta) =z_2,\text{ and}&&
\blop(\a_3) =w_1,
\end{align*}satisfies $$\blop(n_1,n_2,n_3,\lambda_{11}\a_{1}\beta+\lambda_{21}\a_{2}\beta+\lambda_{31}\a_{3},\\
 \lambda_{12}\a_{1}\beta+\lambda_{22}\a_{2}\beta+\lambda_{32}\a_{3})
=\im \la b_1\ra.$$
Thus, the resolution $\B$ is constructed from the resolution (\ref{6.1.1}) by linkage. 
\end{example}

\begin{theorem} 
\label{6.1}
The resolution $\Q$ of Definition~{\rm\ref{K}} may be obtained from the resolution $\B$ of  Definition~{\rm\ref{5.3}} by linkage.\end{theorem}

\begin{proof}
Start with the resolution $\B$ of Definition~\ref{5.3} (and the data and notation of \ref{5.1} and \ref{5.2}).
Fix a homomorphism $\ell:F^* \to G$.  
Define $$c_1=\bmatrix \ell\\ \id\endbmatrix : F^*\to B_1=G\p F^*.$$
Recall, from Proposition~\ref{7.3},
that if $\ell$ is sufficiently general, then the image of $p\circ b_1\circ c_1$ 
is a grade three complete intersection ideal in $R$. 
Define $R$-module homomorphisms $c_2$ and $c_3$: 
\begin{equation}\label{c}\xymatrix{
0\ar[r]& \bw^3F^* \ar[rr]^{p\circ \Br_1\circ c_1}\ar[d]^(.4){c_3}&&
 \bw^2 F^* \ar[rr]^(.55){p\circ \Br_1\circ c_1}\ar[d]^(.4){c_2}&&F^*  \ar[rr]^(.5){p\circ\Br_1\circ c_1}\ar[d]^(.4){c_1}&&R\ar[d]^(.4){p^{-1}}
\\
0\ar[r]& B_3 \ar[rr]^{\Br_3}&&
B_2 \ar[rr]^{\Br_2}&&B_1 \ar[rr]^(.5){\Br_1}&&B_0}\end{equation}
by 
$$c_3(\mff_3)=\bmatrix 
-p((\bw^2\ell)(\psi)\t \mff_3)\cdot p(\chi_G\t \chi_{F^*})\cdot \chi_{G^*}\t \chi_F\\
-z_2\cdot p\left(
\sum_i[\phi(E_i)\w \ell (E_i^*)]\t \mff_3)\right) \cdot p(\chi_G\t \chi_{F^*})\cdot 
  \chi_{G^*}\t\chi_F\\
-z_2\cdot p(\zeta\t \mff_3)
\cdot  p(\chi_G\t \chi_{F^*})\cdot \chi_{G^*}\t  \chi_F\\
\hline
+p(\chi_G\t \chi_{F^*})\cdot p\Big([\ts\bw^2(\phi\circ \ell^*)]
(\chi_{G^*})\t \mff_3\Big)
\cdot \chi_F\\
-p(\zeta\t \mff_3)\cdot p\left(\sum_i[\phi(E_i)\w \ell(E_i^*)]\t \chi_{F^*}\right) \cdot \chi_F\\
- p(\zeta\t \mff_3)
\cdot p(\zeta\t \chi_{F^*})\cdot \chi_F
\endbmatrix$$
in $B_3=(\bw^2G^*\t \bw^3F)\p\bw^3F$ and
$c_2(\mff_2)$ is equal to$$
p(\chi_G\t \chi_{F^*})\cdot \bmatrix
-[(\bw^2 \ell)(\mff_2)](\chi_{G^*})\cdot \chi_{F}
\\
\hline 
+\Big[\ell \Big(
[\psi(\chi_F)](\mff_2)\Big)\Big](\chi_{G^*})
\\
+z_2\cdot  [\phi(\mff_2(\chi_F))](\chi_{G^*})
\\\hline
-\sum_i[\phi(E_i)\w \ell (E_i^*)](\chi_{G^*})\cdot \mff_2
(\chi_F)
- \ell^*\Big(\big(\phi[\mff_2(\chi_F)]\big) (\chi_{G^*})\Big)\\
-  \zeta(\chi_{G^*})\cdot \mff_2(\chi_F)
\endbmatrix$$
in $B_2=\bw^3 F\p G^*\p F$, for $\mff_i\in \bw^iF^*$. We used $\sum E_i\t E_i^*$  as the canonical element of $F\t F^*$ in the sense of \ref{2.3}. (The map $p\circ b_1\circ c_1:F^*\to R$ is an element of $F$, and $\bw^\bullet F^*$ is a module over $\bw^\bullet F$; hence $p\circ b_1\circ c_1:\bw^iF^*\to \bw^{i-1}F^*$ is a meaningful $R$-module homomorphism.) It is not difficult to check that 
(\ref{c}) is a map of complexes, denoted $c:\bw^\bullet F^*\to \B$. The theory of linkage guarantees that the mapping cone of 
the dual of 
 $c$ (denoted $\Tot(c^*)$) is a resolution of $$\frac R{\im(p\circ b_1\circ c_1):\im b_1}.$$ One removes a split exact summand from $\Tot(c^*)$ in order to obtain $\Q$. \end{proof}

\section{Rigidity}\label{Rigidity}

In this section we prove that if $\B$ and $\Q$ are built using variables over a field $\kk$, then the completions of $\HH_0(\B)$ and $\HH_0(\Q)$ are rigid $\kk$-algebras. Artin \cite{Artin} and Hartshorne \cite{Ha10} are excellent references for the basic definitions and results about  deformation theory; the primary source is Lichtenbaum and Schlesinger \cite{LS67}.
 Let 
$A$ be a Noetherian 
 algebra over a field $\kk$ and let ${\pmb \ell}$ be the ring of dual numbers ${\pmb \ell}=\kk[\e]/(\e^2)$. View $\kk$ as an ${\pmb \ell}$-algebra by way of the natural quotient map ${\pmb \ell}\to {\pmb \ell}/(\e)=\kk$.
 A {\it first order infinitesimal 
deformation} of $A$ over $\kk$ 
 is a flat homomorphism $\eta$ of  
$\kk$-algebras $$\eta:{\pmb \ell}\to B$$ such that there exists an isomorphism $\phi$ of $\kk$-algebras  
$\phi:B\t_{\pmb \ell} \kk  
\to A$.  
 The deformation $\eta$ is called {\it trivial} if there exists an isomorphism of ${\pmb \ell}$-algebras $$\Phi:B  
\to A \t_{\kk} {\pmb \ell}$$ 
 such that $\Phi\t_{\pmb \ell} \id_{\kk} = \phi$.
The ring $A$ is a  
{\it rigid} $\kk$-algebra if every first order infinitesimal deformation of $A$ over $\kk$ is trivial.

\begin{chunk}
\label{8.1}Let $R = \kk[[X]]$ be a power series ring over a field $\kk$,  $I$ be an ideal of  $R$,  $A$ be the quotient ring 
$A=R/I$, and $\omega_A$ be the canonical module of $A$.\end{chunk}

Recall the upper cotangent functors
$T^i(-, -)$  
as 
defined by Lichtenbaum and Schlessinger in 
\cite{LS67}. In this language,  
$T^0(A/\kk, A)$ is the module of $\kk$-derivations
$\operatorname{Der}_{\kk}(A, A)$, 
$T^1(A/\kk, A)$
corresponds to the space of isomorphism classes of first-order infinitesimal deformations
of $A$ over $\kk$, and $T^2(A/\kk, A)$ contains the obstructions for lifting infinitesimal
deformations of $A$. Therefore $A$ is  rigid over $\kk$ if $T^1(A/\kk, A) = 0$
and {\it unobstructed} if $T^2(A/\kk, A) = 0$. If  
$I$ is generically a complete intersection
and perfect,  then $A$ is called 
{\it strongly unobstructed} if the twisted conormal module $I/I^2\t_A \omega_A$ is Cohen-Macaulay. 
It is shown in \cite{He80} that if $A$ is strongly unobstructed, then A is unobstructed. 
An important tool for proving results about deformations is the fact, established in \cite{LS67}, that 
 there exists a natural
embedding $$\xymatrix{\Ext^1_A(I/I^2,A)\ar@{^(->}[r] &T^2(A/k, A)},$$ which is an isomorphism whenever $I$ is
generically a complete intersection.

The deformation theory of rings defined by licci
ideals (see Definition~\ref{Feb11})
is especially well-behaved.
Retain the setup of \ref{8.1}. It is shown in \cite{BU83} that
 the depth of the twisted conormal module $I/I^2\t_A \omega_A$ and the depth of the normal module $\Hom_A(I/I^2,A)$ are invariants of the linkage class of $A$.
Thus, the results of \cite{BU83} reprove the fact, established in Buchweitz's thesis \cite{Bu81}, that 
 the Cohen-Macaulay property of $I/I^2\t_A\omega_A$ is preserved under linkage, for ideals $I$ which generically are  complete intersections.
Fix an ideal $I$ of $R$ which is licci and is generically a complete intersection. It follows from the above discussion, that $I$ is strongly unobstructed, every first order infinitesimal deformation of $A=R/I$ may be lifted without restriction, and $A$ is isomorphic to $\widetilde{A}/(\underline{a})$ for some rigid complete local $\kk$-algebra $\widetilde{A}$ and some regular $\widetilde{A}$-sequence $\underline{a}$.

If $A$ and $B$ are complete local $\kk$-algebras, then Herzog \cite{He80} writes $A \sim B$ if there is
a third such algebra $C$ containing regular sequences $\underline{x}$ and $\underline{y}$ such that $A\cong C/(\underline{x})$
and $B = C/(\underline{y})$. This relation becomes an equivalence when the algebras involved
are strongly unobstructed. In particular, $\sim$ is an equivalance relation 
when the rings involved are  defined by licci ideals which 
 generically  are complete intersections.

\begin{example}
\label{7.2}
If the ring $R_0$ is a field $\kk$, then the ideal $(\delta_0,\delta_1,\delta_2,\delta_3)$ of Example~\ref{AltArg} defines a rigid $\kk$-algebra. Indeed, the determinantal ideal $I=(\delta_1,\delta_2,\delta_3)$ of the power series ring $\widehat{R}=\kk[[\{\lambda_{i,j}\}]]$ has been shown to define a rigid quotient many times; see for example, \cite[pg.~685]{Sc77},   \cite{DG79}, \cite[Satz 3.4d]{He80}, or \cite[between 4.2 and 4.3]{KM-def-and-link}. Furthermore, if $R'=\widehat{R}[[\delta_0]]$ and $I'$ is the ideal $(IR',\delta_0)$ of $R'$, then $R'/I'=\widehat{R}/I\widehat{R}$ is also a rigid $\kk$-algebra. 
\end{example}

Our main tool for proving rigidity is the following result about the passage of rigidity across semi-generic linkage. The result appears as \cite[Thm.~4.2]{KM-def-and-link}.

\begin{theorem} \label{KM}  Let 
$\kk$ be a field, $S$ be a power series ring in a finite number of variables  over $\kk$, and $\bf{a}$ be a $1 \times  n$ matrix whose entries
generate a perfect grade $g$ ideal $I$ of $S$.
Suppose $S /I$ is a rigid $\kk$-algebra. Let
$X$ be the $n \times g$ matrix 
$$ X=\bmatrix \begin{array}{c|c}I_t&0\end{array}\\\hline Y\endbmatrix,$$
where $Y$ is an $(n - t) \times g$ matrix of indeterminates for some $t$ with $0\le t \le g$, and $I_t$ is the $t\times t$ identity matrix. Let $K$ be the ideal in $R = S[[Y]]$ generated by 
the entries of the product matrix ${\bf a}X$. 
 If $J$ is  linked to $IR$ by $K$ in $R$, then $R/J$ is a rigid $\kk$-algebra.\end{theorem}

\begin{remark} 
The statement of Theorem \ref{KM} in \cite{KM-def-and-link} includes the hypothesis ``Assume that $K$ is a complete intersection''; however, Proposition~\ref{7.3} 
ensures that the hypothesis is automatically satisfied. \end{remark}
\begin{corollary}
\label{rigid}
If the complex  
$\B$ of Definition 
{\rm\ref{5.3}} is  
built using variables over a field $\kk$,
in the sense of  
{\rm\ref{5.4}}, then the completion of $\HH_0(\B)$ is a rigid $\kk$-algebra.
\end{corollary}

\begin{proof}
Let $R_0$ be the field $\kk$;
$$\widehat{R}=\kk[[\{\lambda_{ij}|1\le i\le 3,\ 1\le j\le 2\}\cup\{\delta_0,\a_1,\a_2,\a_3,\beta\}]]$$ be the completion of the polynomial ring of (\ref{6.1.0}) in the $\m$-adic topology, where $\m$ is generated by all of the variables; 
and $\Ring$ be the power series ring $\Ring=\widehat{R}[[\beta_1,\beta_2]]$.
We consider three ideals of $\Ring$:
$$K:I, \quad K':I,\quad \text{and}\quad K'':I',$$  
where $I=(\delta_0,\delta_1,\delta_2,\delta_3)$ is exactly as described in Example~\ref{AltArg}, 
$I'=(\delta_0,\delta_1',\delta_2',\delta_3')$, where $\delta_1',\delta_2',\delta_3'$ are the signed maximal minors of
$$\bmatrix \lambda_{11}&\lambda_{12}\\\lambda_{21}&\lambda_{22}\\\lambda_{31}-\beta_1\lambda_{11}-\beta_2\lambda_{21}&\lambda_{32}-\beta_1\lambda_{12}-\beta_2\lambda_{22}\endbmatrix,$$
$K$ is generated by the entries of the matrix product
$[\delta_0,\delta_1,\delta_2,\delta_3]X$
for $$X=\bmatrix \a_1&\a_2&\a_3\\1&0&0\\0&1&0\\0&0&\beta\endbmatrix,$$ 
$K'$ is generated by the entries of the matrix product
$[\delta_0,\delta_1,\delta_2,\delta_3]X'$
for $$X'=\bmatrix \a_1&\a_2&\a_3\\1&0&0\\0&1&0\\\beta_1&\beta_2&\beta\endbmatrix,$$ 
and 
$K''$ is generated by the entries of the matrix product
$[\delta_0,\delta_1',\delta_2',\delta_3']X$.

According to Example~\ref{AltArg}, the quotient $\Ring/(K:I)$ is equal to 
$$(\HH_0(\B)\t_R\widehat{R})[[\beta_1,\beta_2]];$$thus, $\HH_0(\B)\t_R\widehat{R}$ is a rigid $\kk$-algebra provided $\Ring/(K:I)$ is a rigid $\kk$-algebra. The quotient $\Ring/(K':I)$ is a rigid $\kk$-algebra by Example~\ref{7.2} and Theorem~\ref{KM}. The $\kk$-algebra isomorphism which sends 
$$\lambda_{3j}\mapsto \lambda_{3j}-\beta_1\lambda_{1j}-\beta_2\lambda_{1j}, \quad\text{for $1\le j\le 2$},$$and is the identity map on all of the other variables, induces a $\kk$-algebra isomorphism 
$$\Ring/(K:I) \to \Ring/(K'':I').$$We complete the proof by showing the the ideals $(K'':I')$ and $(K':I)$ of $\Ring$ are equal. Observe that
$$\delta_1'=\delta_1+\beta_1\delta_3,\quad \delta_2'=\delta_2+\beta_2\delta_3,\quad\text{and}\quad \delta_3'=\delta_3.$$ It follows that the ideals $I'$ and $I$ are equal and the matrices
$$\bmatrix \delta_0&\delta_1'&\delta_2'&\delta_3'\endbmatrix X\quad\text{and}\quad \bmatrix \delta_0&\delta_1&\delta_2&\delta_3\endbmatrix X'$$ are equal.
\end{proof}

\begin{corollary}
\label{rigid2}
If the complex $\Q$ of
Definition {\rm\ref{K}} 
is  
built using variables over a field $\kk$,
in the sense of {\rm\ref{3.5.1}}, 
 then the completion of $\HH_0(\Q)$ is a rigid $\kk$-algebra.\end{corollary}

\begin{proof}
 Apply Theorem~\ref{6.1}, Theorem~\ref{KM}, and Corollary~\ref{rigid}.\end{proof}


\begin{thebibliography}{99}
\bibitem{Artin} M.~Artin, {\em Lectures on Deformations of Singularities}, Tata Institute, Bombay, 1976. \newline
{\tt http://www.math.tifr.res.in/$\sim$publ/ln/tifr54.pdf}

\bibitem{Br87}  A.~Brown, {\em A structure theorem for a class of grade three perfect ideals}, J. Algebra {\bf 105} (1987),  308--327.

\bibitem{Br84}W.~Bruns, {\em The existence of generic free resolutions and related objects}, Math. Scand. {\bf 55} (1984),  33--46.

\bibitem{BV} W.~Bruns and U.~Vetter, {\em   Determinantal rings}, Lecture Notes in Mathematics, {\bf 1327}, Springer-Verlag, Berlin, 1988.


\bibitem{Bu81}R.-O.~Buchweitz, {\em Contributions \`a la th\'eorie des singularit\'es}, Thesis, University of Paris, 1981.


\bibitem{BU83} R.-O.~Buchweitz and B.~Ulrich, {\em Homological properties which are invariant under linkage}, preprint, 1983.



\bibitem{BE75} D.~Buchsbaum and D.~Eisenbud, {\em
Generic free resolutions and a family of generically perfect ideals}, 
Advances in Math {\bf 18} (1975),  245--301.

\bibitem{BE77}D.~Buchsbaum and D.~Eisenbud, {\em
Algebra structures for finite free resolutions, and some structure theorems for ideals of codimension {\rm3}},
Amer. J. Math. {\bf 99} (1977),  447--485. 

\bibitem{Bu68} L.~Burch, {\em On ideals of finite homological dimension in local rings}, Proc. Cambridge Philos. Soc. {\bf 64} (1968), 941--948.

\bibitem{CKLW} E.~Celikbas, J.~Laxmi, W.~Kra\'skiewicz, and J.~Weyman, {\em The family of perfect ideals of codimension $3$, of Type $2$ with $5$ generators}, Proc. Amer. Math. Soc., (to appear). 

\bibitem{CVW17}L.~Christensen, O.~Veliche, and J.~Weyman {\em Free resolutions of Dynkin format and the licci property of grade $3$ perfect ideals}, {\tt https://arxiv.org/pdf/1712.04016.pdf} Math. Scand. (to appear).

\bibitem{DG79} P.~de Carli and  S.~Gabelli, 
Una classe di algebre rigide,
Ann. Univ. Ferrara Sez. VII (N.S.) {\bf 25} (1979), 69--74. 

\bibitem{G80} E.~S.~Golod, {\em A note on perfect ideals}, Algebra (A. I. Kostrikin, ed.), Moscow State Univ. Publishing
House, (1980), 37--39.


\bibitem{Ha10}R.~Hartshorne, {\em
Deformation theory},
Graduate Texts in Mathematics, {\bf 257} Springer, New York, 2010.


\bibitem{He80} J.~Herzog, {\em Deformationen von Cohen-Macaulay Algebren},  J. Reine Angew. Math. {\bf 318} (1980), 83--105.



\bibitem{H75} M.~Hochster, {\em
Topics in the homological theory of modules over commutative rings},
Expository lectures from the CBMS Regional Conference held at the University of Nebraska, Lincoln, Neb., June 24--28, 1974, {\bf 24},  American Mathematical Society, Providence, R.I., 1975. 

\bibitem{HE} M.~Hochster and J.~Eagon, {\em Cohen-Macaulay rings, invariant theory, and the generic perfection of determinantal loci}, Amer. J. Math. {\bf 93} (1971), 1020--1058. 

\bibitem{Ku93}A.~Kustin, 
{\em Ideals associated to two sequences and a matrix}, 
Comm. Algebra {\bf 23} (1995),  1047--1083. 

\bibitem{KM-def-and-link}A.~Kustin and M.~Miller, {\em
Deformation and linkage of Gorenstein algebras},
Trans. Amer. Math. Soc. {\bf 284} (1984), 501--534. 


\bibitem{LS67}S.~Lichtenbaum and M.~Schlessinger, {\em
The cotangent complex of a morphism},
Trans. Amer. Math. Soc. {\bf 128} (1967), 41--70. 


\bibitem{Na}  M.~Nagata, {\em Local rings}, Interscience Tracts in Pure and Applied Mathematics, No. 13 Interscience Publishers a division of John Wiley \& Sons,  New York-London, 1962.

\bibitem{PS} C.~Peskine and L.~Szpiro, {\em Liaison des vari\'et\'es alg\'ebriques. I.}, Invent. Math. {\bf 26} (1974), 271--302.

\bibitem{Sc77} M.~Schaps, {\em
Deformations of Cohen-Macaulay schemes of codimension {\rm2} and non-singular deformations of space curves},
Amer. J. Math. {\bf 99} (1977),  669--685. 

\bibitem{W90}  J.~Weyman, {\em On the structure of free resolutions of length $3$}, J. Algebra {\bf 126} (1989),  1--33.

\bibitem{W18} J.~ Weyman, {\em Generic free resolutions and root systems}, Ann. Inst. Fourier (Grenoble) {\bf 68} (2018),  1241--1296.


\end{thebibliography}
\end{document}